\documentclass[11pt]{amsart}

\newtheorem{theorem}{Theorem}[section]

\newtheorem{lemma}[theorem]{Lemma}

\newtheorem{proposition}[theorem]{Proposition}

\newtheorem{example}[theorem]{Example}

\theoremstyle{definition}
\newtheorem{remark}[theorem]{Remark}

\usepackage[colorlinks, bookmarks=true]{hyperref}
\usepackage{color,graphicx,shortvrb}
\usepackage[active]{srcltx} 

\usepackage{enumerate}
\usepackage{amsmath}
\usepackage{amssymb}
\usepackage[latin1]{inputenc}

\def\J#1#2#3{ \left\{ #1,#2,#3 \right\} }

\def\11{\textbf{$1$}}

\begin{document}

\numberwithin{equation}{section}

\title{Linear isometries between real JB$^*$-triples and C$^*$-algebras}

\author[M. Apazoglou]{Maria Apazoglou}
\address{Queen Mary, University of London, London, United Kingdom}
\email{m.apazoglou@qmul.ac.uk}

\author[A.M. Peralta]{Antonio M. Peralta}
\address{Departamento de An{\'a}lisis Matem{\'a}tico, Universidad de Granada,\\
Facultad de Ciencias 18071, Granada, Spain}
\email{aperalta@ugr.es}

\thanks{Second author partially supported by the Spanish Ministry of Science and Innovation,
D.G.I. project no. MTM2011-23843, and Junta de Andaluc\'{\i}a grants FQM0199 and
FQM3737.}

\thanks{Part of the results of this paper are contained in the first author's PhD
thesis \cite{apazoglou} presented to Queen Mary, University of
London, funded by EPSRC. The first author wishes to thank Professor
Cho-Ho Chu for his valuable guidance and support.}
\date{}
 \maketitle

\begin{abstract}
Let $T: A\to B$ be a (not necessarily surjective) linear isometry between two real JB$^*$-triples.
Then for each $a\in A$ there exists a tripotent $u_a$ in the bidual, $B'',$ of $B$ such that
\begin{enumerate}[$(a)$]
\item $\{u_a,T(\{f,g,h\}),u_a\}=\{u_a,\{T(f),T(g),T(h)\},u_a\}$, for all $f,g,h$ in the real JB$^*$-subtriple, $A_a,$ generated by $a$;
\item The mapping $\{u_a,T(\cdot),u_a\} :A_a\rightarrow B''$ is a linear isometry.
\end{enumerate}
Furthermore, when $B$ is a real C$^*$-algebra, the projection $p=p_a=
u_a^* u_a$ satisfies that $T(\cdot)p :A_a\rightarrow B''$ is an
isometric triple homomorphism. When $A$ and $B$ are real
C$^*$-algebras and $A$ is abelian of real type, then there exists a
partial isometry $u\in B''$ such that the mapping $T(\cdot)u^*u
:A\rightarrow B''$ is an isometric triple homomorphism. These
results generalise, to the real setting, some previous contributions
due to C.-H. Chu and N.-C. Wong, and C.-H. Chu and M. Mackey in 2004
and 2005. We give an example of a non-surjective real linear
isometry which cannot be complexified to a complex isometry, showing
that the results in the real setting can not be derived by a mere
complexification argument.
\end{abstract}

\section{Introduction}
Despite many recent applications of real C$^*$-algebras in operator
theory \cite{chu11, didenko_silbermann}, in JB$^*$-triples and
infinite dimensional geometry \cite{chu0, chu11, chu10,
edwards_ruttimann, FerMarPe, isidro_palacios_kaup, peralta,
peralta_stacho}, the theory of real C$^*$-algebras has not been
fully developed. In this paper, we study the linear geometry of real
C$^*$-algebras and in particular, the not necessarily surjective
linear isometries between real C$^*$-algebras, thereby extending the
isometry results in \cite{chu9}, \cite{chu2} and \cite{kadison} for
complex C$^*$-algebras.

A  celebrated result of Kadison \cite{kadison} states that a
surjective linear isometry $T:A\rightarrow B$ between complex
C$^*$-algebras is necessarily a Jordan triple isomorphism, that is,
\begin{equation*}
T(ab^*c+cb^*a)=T(a)T(b)^*T(c)+T(c)T(b)^*T(a)\qquad(a,b,c\in A).
\end{equation*}
This result need not be true when the hypothesis of $T$ being
surjective is not assumed. Nevertheless, it has been shown by C.-H.
Chu and N.-C. Wong \cite{chu10} that even if $T$ is non-surjective,
for each $a\in A$, there is a largest projection $p_a\in B^{**}$,
such that \begin{equation*} T(\cdot)p_a:A_a\rightarrow B^{**}
\end{equation*}
is an isometric Jordan triple homomorphism, where $A_a$ denotes the
JB$^*$-subtriple of $A$ generated by $a$. In \cite{chu9}, C.-H. Chu
and M. Mackey have further extended this result to the case in which
$A$ and $B$ are JB$^*$-triples, showing, among other things, that,
in this case, for each $a\in A$ there exists a tripotent $u_a\in
B^{**}$ such that
\begin{equation*}
\{u_a,T(\{f,g,h\}),u_a\}=\{u_a,\{T(f),T(g),T(h)\},u_a\} \ \ (f,g,h\in A_a)
\end{equation*} and the mapping $\{u_a,T(\cdot),u_a\}$
is an isometry.

Our objective is not only to extend these results to real
C$^*$-algebras but also to clarify the relationship between the
projection $p_a$ and the tripotent $u_a$ which was not given in
\cite{chu9}. Indeed, we show that both results above are valid for
real C$^*$-algebras and moreover, we have $p_a=u_a^*u_a$. Prior to
this result, we establish that when $A$ is a real C$^*$-algebra
which admits a complex character of real type, then for every not
necessarily surjective linear isometry $T$ from $A$ to another
C$^*$-algebra $B$, there exists a minimal projection $p$ in $B''$ satisfying:
\begin{enumerate}[$(a)$] \item $T\{a,b,c\} p= \{T(a),T(b),T(c)\} p$ and $p T(a)^* T(b) = T(a)^* T(b) p,$  for all $a,b,c$ in $A$;
\item $T(\cdot) p : A \to B''$ is a non-zero triple homomorphism
\end{enumerate} (see Proposition \ref{p linear contractions} and Theorem \ref{t Chu Wong Prop 4.3}). When $A$ does not admit
complex characters of real type, that is, when $A$ is of complex type we can find an example of a linear surjection $T$
from $A$ into another real C$^*$-algebra such that $T(\cdot) p=0$ whenever $T(\cdot) p$ is a triple homomorphism,
$T\{a,b,c\} p= \{T(a),T(b),T(c)\} p$ and $p T(a)^* T(b) = T(a)^* T(b) p,$  for all $a,b,c$ in $A$ (compare Example
\ref{examp real/complex type to ChuWong Thm 3.10}).

We culminate the paper showing that every non-surjective
linear isometry between real JB$^*$-triples $T:A \to B$, although need not be a triple homomorphism,
locally reduces to a triple homomorphism, that is, for each $a\in A$, there exists a tripotent $u\in B''$ such that
\begin{enumerate}[$(a)$]
\item $\{u,T(\{f,g,h\}),u\}=\{u,\{T(f),T(g),T(h)\},u\}$, for all $f,g,h$ in the real JB$^*$-subtriple generated by $a$;
\item The mapping $\{u,T(\cdot),u\} :A_a\rightarrow B''$ is a linear isometry.
\end{enumerate}
Furthermore, when $B$ is a real C$^*$-algebra, the projection $p= u^* u$ satisfies that
$T(\cdot)p :A_a\rightarrow B''$ is an isometric triple homomorphism (compare Theorem \ref{t 2iso}).

Our results cannot be achieved by simple
complexification since, as shown later, the complexification of a
non-surjective real isometry need not be, in general, an isometry.
We develop some new techniques to accomplish our task.

\section{Preliminaries: Jordan structures in real C$^*$-algebras}

Let $A$ be a real Banach*-algebra. Then $A$ is called a \emph{real
C$^*$-algebra} if $\|a^*a\|=\|a\|^2$ and $1+a^*a$ is invertible for all
$a\in A$. If $A$ is non-unital, then we require that $1+a^*a$ is
invertible in the unit extension $A\oplus\mathbb{R} 1$. Equivalently,
a real Banach *-algebra $A$ is a real C$^*$-algebra if, and only if, it is
isometrically *-isomorphic to a norm-closed real *-algebra of bounded
operators on a real Hilbert space (cf. \cite[Corollary 5.2.11]{li1}).
Trivially, every (complex)
C$^*$-algebra is a real C$^*$-algebra when scalar multiplication is
restricted to the real field. Let $A$ be a real C$^*$-algebra. We
denote by $A'$ the dual space of $A$. The complexification
$A_c=A\oplus iA$ is a C$^*$-algebra in a suitable norm such that $A$
identifies with a real closed *-subalgebra of $A_c$. The dual
space of a complex C$^*$-algebra $B$ will be denoted by $B^*$.

Throughout the paper, given a real or complex C$^*$-algebra $A$, $A_{sa}$
will stand for the set of all self-adjoint elements in $A$.

We shall now survey some Jordan structures associated with real
and complex C$^*$-algebras. A (complex)
\emph{JB$^*$-triple} is a complex Banach space $A$ equipped with a
triple product $\{\cdot,\cdot,\cdot\}:A\times A\times A\rightarrow
A$ which is linear and symmetric in the outer variables, conjugate
linear in the middle one and satisfies the following conditions:
\begin{enumerate}
\item (Jordan identity) for $a,b,x,y,z\in A$,
\begin{equation}\label{jordanid}\{a,b,\{x,y,z\}\}=\{\{a,b,x\},y,z\}
-\{x,\{b,a,y\},z\}+\{x,y,\{a,b,z\}\};\end{equation}
\item $D(a,a):A\rightarrow A$ is an hermitian
linear operator with non-negative spectrum, where $D(a,b)(x)=\{a,b,x\}$ with $a,b,x\in
A$;
\item $\|\{x,x,x\}\|=\|x\|^3$ for all $x\in A$.
\end{enumerate}
A complex subspace $B$ of $A$ is called a \emph{subtriple} if it is
closed with respect to the triple product, that is, $x,y,z\in B$
implies $\{x,y,z\}\in B$. A real subtriple of $A$ is a real subspace
of $A$ which is closed with respect to the triple product. We define
a \emph{real JB$^*$-triple} to be a closed real subtriple of a
JB$^*$-triple. When restricted to real scalar multiplication, every
(complex) JB$^*$-triple is also a real JB$^*$-triple.
It has been shown in \cite{isidro_palacios_kaup} that a
real JB$^*$-triple can be complexified to become a JB$^*$-triple.
Furthermore, every real JB$^*$-triple $B$ is a real form of a complex
JB$^*$-triple, that is, there exist a (complex) JB$^*$-triple $B_{c}$
and a conjugate linear isometry $\tau: B_c\rightarrow B_c$ of period 2
such that $B=\{b\in B_c\::\:\tau(b)=b\}$. We can actually identify $B_{c}$ with the
complexification of $B$.

A real C$^*$-algebra $A$ is a real JB$^*$-triple in the
triple product:
\begin{equation}\label{eq triple product in C*-algebras}
\{a,b,c\}=\frac{1}{2}(ab^*c+cb^*a)\qquad(a,b,c\in A).
\end{equation}
We shall always assume the above triple product in a real or complex
C$^*$-algebra. \\
An element $e$ in a complex or real JB$^*$-triple $A$ is called
\emph{tripotent} if $\{e,e,e\}=e$.
Let $A$ be a real JB$^*$-triple. For an element $a\in A$, we define
the operators $D(a):A\rightarrow A$ and $Q(a):A\rightarrow A$ by
\begin{equation*} D(a)(x)=\{a,a,x\}\;\textrm{and}\;
Q(a)(x)=\{a,x,a\}\qquad(x\in A).
\end{equation*}
The Peirce projections associated with a tripotent $e\in A$ are given
by \begin{align*}
P_2(e)&=Q(e)^2\\
P_1(e)&=2(D(e)-Q(e)^2)\\
P_0(e)&=I-2D(e)+Q(e)^2.
\end{align*}
The space $A_j(e)=P_j(e)(A)$ is called the Peirce-$j$-space
associated to the tripotent $e$. Thus, each tripotent $e$ in a real or complex JB$^*$-tripe, $A$,
defines a decomposition of $A$ in terms of Peirce subspaces $$ A =
A_{0} (e) \oplus A_{1} (e) \oplus A_{2} (e).$$ Another decomposition
can be given in terms of the eigensubspaces of $Q(e)$, $$A= A^{1}
(e) \oplus A^{-1} (e) \oplus A^{0} (e)$$ where $A^{k} (e) := \{ x\in
A : Q(e) (x) :=\{ e,x,e\} = k x \}$ is a real Banach subspace of
$A$. Furthermore, the following identities and rules hold:
\begin{align*}
&A_{2} (e) = A^{1} (e) \oplus A^{-1} (e),\\
&A_{1} (e)\oplus A_{0} (e) = A^{0} (e) \\
&\{A^{i} (e),A^{j}(e),A^{k} (e)\} \subseteq A^{i j k} (e), \hbox{
whenever } i j k \ne 0.
\end{align*}

A real JB$^*$-triple which is also a dual Banach space is called a
{real} (respectively, complex) \emph{JBW$^*$-triple}. Every real or
complex JBW$^*$-triple has a (unique) predual (cf. \cite{BarTi},
\cite[Corollary 3.2]{edwards_ruttimann}, \cite{Horn} and
\cite{MarPe}) and its triple product is separately weak$^*$
continuous (see \cite{MarPe}).

Given a tripotent $e$ in a real or complex JB$^*$-triple
$A$, the Peirce space $A_2 (e)$ is a real JB$^*$-algebra and $A^{1}
(e)$ is a JB-algebra with product and involution given by $a\circ_e b
:= \{ a,e,b\}$ and $a^{\sharp_e} := \{ e,a,e\}$, respectively. It can be easily seen
that, for each $a\in A$, the expression \begin{equation}\label{equation positive element}
P_2 (e) \J aae = \J {P_2 (e) (a)}{P_2 (e) (a)}e+\J {P_1 (e) (a)}{P_1 (e) (a)}e
\end{equation} $$= P^1 (e) \J aae $$
defines a positive element in the JB-algebra $A^1 (e)$ (see, for example the arguments given in \cite[Lemma 2.1]{BuFerMarPe},
which remain valid in the real setting).

Two elements $a,b$ in a real or complex JB$^*$-triple $A$ are said to be
\emph{orthogonal} (written $a\perp b$) if $L(a,b) =0$. By Lemma 1 in \cite{BurFerGarMarPe}
(whose proof is also valid for real JB$^*$-triples), two elements $a,b$ in $A$ are orthogonal if,
and only if, $\J aab =0$ (equivalently, $\J bba =0$). It follows from the Peirce rules that, for each tripotent
$e$ in $A$, the sets $A_0(e)$ and $A_2 (e)$ are
orthogonal.

In a real or complex JB$^*$-triple, $A$, a partial order relation, $\leq$, can be
defined in the set of all tripotents elements in $A$ given by $e\leq
f$ if, and only if, $f-e$ is a tripotent in $A$ with $f-e\perp e$.
A non-zero tripotent $e$ is called \emph{minimal} whenever $A^{1}
(e) = \mathbb{R} e$ (since in the complex case $A^{-1}(e)= i
A^{1}(e)$, this definition is equivalent to say $A_{2} (e) =
\mathbb{C} e$). However, in the real setting the dimensions of
$A^1(e)$ and $A^{-1} (e)$ are not correlated as there exist examples
of real Cartan factors $A$ containing a minimal tripotent $e$
satisfying dim$(A^{-1} (e))=\infty$. Another illustrating example is
$A=\mathbb{R}, \mathbb{C}$ or $\mathbb{H}$ and $e= 1$, then $e$ is
minimal and dim$(A^{-1}(e))=0,1$ and $3,$ respectively. In general,
given a minimal tripotent $e$ in a real JB$^*$-triple $A$, $\{
e,A,e\} = A_2 (e)$ need not coincide with $\mathbb{R} e = A^{1}
(e)$.

In \cite{peralta_stacho} a tripotent $e$ is called minimal when it is
minimal with respect to the partial order relation $\leq$, that is,
$e$ is minimal iff $e\geq f\neq 0$ implies $e=f$. To avoid
confusion, let us call these kind of tripotents \emph{order minimal
tripotents}. Clearly every minimal tripotent is order-minimal,
however, the reciprocal statement is not, in general, true. For
example, the unit element $1$ in $C[0,1]$ (real or complex valued)
is order minimal because there exists no tripotent $e\in C[0,1]$
satisfying $1>e$. When $A$ is a real JBW$^*$-triple minimal and
order minimal tripotents in $A$ coincide (cf. Proposition 2.2 and
Remark 2.3, \cite{peralta_stacho}).

We recall that a minimal tripotent $e$ in a real JB$^*$-triple $E$
is called \emph{reduced} whenever $E_{2} (e) = \mathbb{R} e$
(equivalently, $E^{-1} (e) = 0$). The real JB$^*$-triple $E$ is said to be
\emph{reduced} when every minimal tripotent $e\in E$ is reduced (cf. \cite[11.9]{Loos2}).

A real W$^*$-algebra is a real C$^*$-algebra $A$ having
a predual, $A_{_{'}}$, in which case, $A_{_{'}}$ is unique up to linear
isometric isomorphism and the product of $A$ is separately weak$^*$
continuous (cf. \cite{IsRo96}). We note that the second dual $A''$ of a real
C$^*$-algebra $A$ is a real W$^*$-algebra. Moreover, let $A_c$ and $A_c^{**}$ denote the
complexifications of $A$ and its bidual, respectively, then $A_c^{**}=A''\oplus i A''$.
Every W$^*$-algebra is a real JB$^*$-triple with respect to the product given in $(\ref{eq triple product in C*-algebras})$.

An element $p$ in a real C$^*$-algebra $A$ is called a
\emph{projection} if $p=p^*=p^2$. A projection $p$ in $A$ is \emph{minimal} whenever
it is a minimal tripotent in $A$, that is, $A^{1} (p) =\{x\in A : px^* p = x\} =
\mathbb{R} p$. It is well known that, when $A$ is regarded as a real JB$^*$-triple,
tripotents in $A$ correspond to \emph{partial isometries} in $A$ (i.e. elements $u\in A$ with $u^*u$
being a projection in $A$).

The next lemma should be known for experts but we were not able to find an explicit reference.

\begin{lemma}\label{lemma2js}
Let $u$ be a tripotent in a real C$^*$-algebra $A$ regarded as
a real JB$^*$-triple. Then $u$ is minimal if, and only if, $p=u^*u$ (respectively, $q=uu^*$)
is a minimal projections in $A$.
\end{lemma}

\begin{proof} The equivalence follows from the following identities:
$A^{1} (p) = p A_{sa} p,$ $  u A^{1}(p)= A^{1} (u) = \{ x\in A : u x^* u = x\}.$
\end{proof}

Let $A$ be a real W$^*$-algebra. 
Two projections $p$ and $q$ in $A$ are called \emph{orthogonal} if $pq=0$.
A projection $p$ is \emph{majorised} by $q$, i.e. $p\leq q$ in the partial order
of the cone $A_+$, if and only if $p(1-q)=0$. When $A$ is regarded as a JB$^*$-triple
this partial order agrees with the partial order defined in the set of tripotents
(i.e. partial isometries) of $A$.

\begin{lemma}\label{1.4.7}{\rm
(\cite{li1},Proposition 4.3.4)} Let $A\subset B(H)$ be a real
W$^*$-algebra. Then the supremum $\bigvee_{i\in I}p_i$ of a family
$\{p_i\}_{i\in I}$ of projections in $A$ exists with respect to the
ordering induced by the cone $A_+$. Further, we have
\begin{equation*}
\bigvee_{i\in I}p_i:H \rightarrow\overline{[\cup_{i\in I}p_iH]}.
\end{equation*}
is the projection onto the closed linear span $\overline{[\cup_{i\in
I}p_iH]}$ of $\cup_{i\in I}p_iH$.
\end{lemma}

\begin{lemma}\label{lemma3js}
Let $\{p_\phi\}_{\phi\in Q}$ be a family of projections in a real
W$^*$-algebra $W$. Let $p=\bigvee_{\phi\in Q}p_\phi$. If $ap_\phi=0$
for some $a\in W$ and all $\phi\in Q$, then $ap=0$.
\end{lemma}

\begin{proof}
By hypothesis, $a\in W$ annihilates on every element in the linear
span of $\bigcup_{i\in I} p_i (H)$, and by continuity $a p(\xi) =
0$, for every $\xi\in H$.
\end{proof}

If $B$ is a real JB$^*$-triple, then its second dual, $B''$, is also a
real JB$^*$-triple. Moreover, $B$ is the real form of a complex
JB$^*$-triple $B_c$, that is, there is a conjugate linear isometry
$\tau:B_c\rightarrow B_c$ of period 2 such that $B=\{b\in
B_c\::\:\tau(b)=b\}$. Further, the bidual map
$\sigma=\tau^{**}:B_c^{**}\rightarrow B_c^{**}$ is a conjugate
linear isometry of period 2 and $B''=\{b\in
B_c^{**}\::\:\sigma(b)=b\}$. The ``dual'' map
$\widetilde{\tau}:B_c{^*}\rightarrow B_c^{*}$ defined by
\begin{equation*}
\widetilde{\tau}(\phi)(b)=\overline{\phi(\tau(b))}\qquad(\phi \in B_c^*,\,b\in B_c)
\end{equation*} is a conjugate linear isometry of period 2 and the mapping $$(B_c^*)^{\widetilde{\tau}} \to B'$$
$$\varphi \mapsto \varphi|_{B}$$ is a surjective linear isometry.

We observe that if $u$ is a tripotent in
$B_c^{**}$, then $\sigma(u)$ is also a tripotent in $B_c^{**}$. In
fact, letting $U(B_c^{**})$ be the set of all tripotents in
$B_c^{**}$, the set $U(B'')$ of tripotents in $B''$ is
$U(B_c^{**})\cap B''=\{u\in U(B_c^{**})\::\:\sigma(u)=u\}$
\cite{edwards_ruttimann}. For a functional $\phi\in B_c^*$, there
exists a unique tripotent $u_\phi\in B_c^{**}$, called the
\emph{support tripotent} of $\phi$, such that $\phi=\phi\circ
P_2(u_\phi)$ and $\phi|_{P_2(u_\phi)(B_c^{**})}$ is a faithful
normal positive functional on the JBW$^*$-algebra
$P_2(u_\phi)(B_c^{**})$ (cf. \cite{FriRu}). We note that if $u_\phi\in B_c^{**}$ is the
support tripotent of $\phi\in B_c^{*}$, then $\sigma(u_\phi)$ is the
support tripotent of $\sigma_*(\phi)$. On the other hand, if $\phi$
is in $B'$, its complex extension $\phi_c\in B_c^*$ has support
tripotent $u_{\phi_c}\in B_c^{**}$ such that
$\sigma(u_{\phi_c})=u_{\phi_c}$ since $\sigma^{\sharp} (\phi_c)=\phi_c$.
Hence $u_{\phi_c}\in U(B'')$ and we call it the \emph{support
tripotent} of $\phi$ in $B''$, denoted by $u_\phi$ (cf. \cite[Lemma 2.2]{peralta}).
Finally, we note that $\phi$ is an extreme point of the closed unit ball of $A_*$ if
and only if its support tripotent $u_{\phi}$ is a minimal tripotent
in $A$ \cite[Corollary 2.1]{peralta_stacho}. Moreover, in this case
$\phi = \phi P^{1} (u_{\phi})$ (and hence $\phi = \phi Q(u_{\phi})$)
(\cite[Lemma 2.7]{peralta_stacho}). Therefore
$$\phi (x) u_{\phi} =\frac12 \left( Q(u_{\phi})^2 (x) + Q(u_{\phi}) (x)\right)= P^{1} (u_{\phi} ) (x),$$
for every $x\in A$. Based on these observations and the complex
results in \cite[Proposition 1.2]{barton_friedman} we have the
following lemma.

\begin{lemma}\label{lemma1js}{\rm (\cite[Lemma 2.4]{peralta})}
Let $B''$ be the second dual of a real
C$^*$-algebra $B$ and let $\phi\in B'$ with $\|\phi\|=1$. Let $u,v\in
B''$ such that $\phi(u)=\|u\|=1=\phi(v)=\|v\|$ then
\begin{enumerate}
\item $\phi\{x,y,v\}=\phi\{x,y,u\}=\phi\{y,x,u\}$;
\item $\phi\{x,x,u\}\geq0$;
\item $|\phi\{x,y,u\}|^2\leq\phi\{x,x,u\}\phi\{y,y,u\}$ (Cauchy-Schwarz inequality),
\end{enumerate} for all $x,y\in B''$.
\end{lemma}

\begin{lemma}\label{p0space}
Let $B''$ be the second dual of a real C$^*$-algebra $B$ and let
$\phi\in B'$ be an extreme point in the closed unit ball $B'_1$. Let
$N_\phi=\{b\in B''\::\:\phi\{b,b,u\}=0\}$ where $u$ is the support
tripotent of $\phi$. Then $N_\phi=P_0(u)(B'')$.
\end{lemma}

\begin{proof}
By $(\ref{equation positive element})$ and the faithfulness of $\phi|_{(B'')^1 (u)}$,
we have $\{b,b,u\}=0$, that is, $b\perp u$, and hence $b\in P_0(u)(B'')$.
\end{proof}

A (real) linear functional $\varphi$ on a real C$^*$-algebra $A$ is said to be positive when it is hermitian and maps positive elements in $A$ into $\mathbb{R}^{+}_0$ (i.e. $\varphi(a )\geq 0$, for every positive element $a $ in $A$, and $\varphi (b^*) = \varphi (b)$, for every $b$ in $A$).
Following standard notation, given a real C$^*$-algebra $A$, the quasi-state space of $A$,
$\mathcal{Q}(A),$ is defined as the set of all positive $\varphi\in A'$ satisfying $\|\varphi\|\leq 1$ (see, for example \cite[\S 5.2]{li1}). It is known that $\mathcal{Q}(A)$ is a weak$^*$ closed convex subset of $(A')_1$. We shall require later some results on the facial structure of $\mathcal{Q} (A)$.
Having in mind that $A_{sa}$ is a JC-algebra with respect to the canonical Jordan product,
the following Lemma was proved in \cite[Theorem 5.2]{Ne00}.

Finally, let $A$ be a real or complex C$^*$-algebra or a JB-algebra. Following standard notation (see \cite[\S 3.11]{Ped} and \cite{Ed77}), we recall that a projection $p$ in the bidual, $A^{**}$, of $A$ is said to be open if $A^{**}_2 (p) \cap A$ is weak$^*$-dense in $A^{**}_{2}(p)$. The projection $p$ is called {closed} when $1 -p$ is open.

\begin{lemma}\label{l facial structure}{\rm \cite[Theorem 5.2]{Ne00}}
Let $\mathcal{Q}(A)$ be the quasi-state space of a real C$^*$-algebra $A$.
For every norm-closed face $F$ of $\mathcal{Q} (A)$ there exists a (unique) closed
projection $p\in A''$ satisfying that
$F=\{ \varphi\in \mathcal{Q} (A) : \varphi (1-p)=0\}.\hfill\Box$
\end{lemma}

\section{Isometries between real C$^*$-algebras}
Our goal in this section is to show that any linear isometry
$T:A\rightarrow B$, surjective or not, between real C$^*$-algebras $A$
and $B$ reduces \emph{locally} to an ({isometric}) Jordan triple
isomorphism by a projection $p=u^*u\in B''$ for
some partial isometry $u\in B''$.

In our first result we shall study contractive linear projections
between real C$^*$-algebras. In the complex setting, C.-H. Chu and
N.-C. Wong proved in \cite[Proposition 2.2]{chu2} that for every
linear contraction $T$ between (complex) C$^*$-algebras $A$ and $B$,
there exists a largest projection $p\in B^{**}$ such that
\begin{enumerate}[$(a)$] \item $T\{a,a,a\} p= \{T(a),T(a),T(a)\} p$;
\item $p T(a)^* T(a) = T(a)^* T(a) p$ for every $a\in A$.
\end{enumerate} A standard complex polarisation argument implies
\begin{enumerate}[$(a)$] \item $T\{a,b,c\} p= \{T(a),T(b),T(c)\} p$ and $p T(a)^* T(b) = T(a)^* T(b) p$  for all $a,b,c$ in $A$;
\item $T(\cdot) p : A \to B^{**}$ is a triple homomorphism.
\end{enumerate}

In the real setting, the lacking of a standard polarisation identity forces us to make a
slight modification of the argument given by Chu and Wong in \cite[Proposition 2.2]{chu2},
an sketch of the proof is included here for completeness.

\begin{proposition}\label{p linear contractions} Let $T : A \to B$ be a linear contraction between real C$^*$-
algebras. Then there is a largest projection $p$ in $B''$ such that
\begin{enumerate}[$(a)$] \item $T\{a,b,c\} p= \{T(a),T(b),T(c)\} p$ and $p T(a)^* T(b) = T(a)^* T(b) p$  for all $a,b,c$ in $A$;
\item $T(\cdot) p : A \to B''$ is a triple homomorphism.
\end{enumerate} Furthermore, the projection $p$ is closed.
\end{proposition}

\begin{proof} Let us define $$ F_1:=\bigcap_{a,b\in A_1} \Big\{ \varphi \in \mathcal{Q} (B) : \|T(Q(a)b)- Q(T(a))(T(b))\|_{\varphi}=0 \Big\}$$ $$= \bigcap_{a,b\in A_1} \Big\{ \varphi \in \mathcal{Q} (B) : {\varphi} \Big((\Theta_{a,b}) (\Theta_{a,b})^* + (\Theta_{a,b})^* (\Theta_{a,b}) \Big)=0 \Big\},$$ where, for each $a,b\in A,$ $\Theta_{a,b} = T(Q(a)b)- Q(T(a))(T(b)).$ Clearly, $F_1$ is a weak$^*$ closed face of $\mathcal{Q} (B)$ containing zero. Given $a,b\in A_1,$ we define a continuous affine mapping $\Phi_{a,b} : \mathcal{Q} (B) \to \mathcal{Q} (B),$ $\varphi\mapsto \Phi_{a,b} (\varphi) (\cdot) := \varphi (T(b)^* T(a) \cdot T(a)^* T(b)).$ For each natural $n\geq 2$, $$F_n := \bigcap_{a,b\in A_1} \Big(F_{n-1} \bigcap \Phi_{a,b}^{-1} (F_{n-1}) \Big)$$ is a weak$^*$ closed face of $\mathcal{Q} (B)$ containing zero and $F_{n} \subseteq F_{n-1}.$ Therefore $F= \bigcap_{n=1}^{\infty} F_n$ is a weak$^*$ closed face of $\mathcal{Q} (B)$ containing zero, and hence there exists a closed projection $p\in B''$ satisfying that $F= \{\varphi\in \mathcal{Q} (B) : \varphi (1-p)=0\}$ (compare Lemma \ref{l facial structure} or \cite[Theorem 5.2]{Ne00}). In particular, for each $a,b\in A$ and $\varphi \in F$ we have $$0=\Phi_{a,b} (\varphi) (1-p) = \varphi (T(b)^* T(a) (1-p) T(a)^* T(b))$$
$$= \varphi \Big(((1-p) T(a)^* T(b))^* ((1-p) T(a)^* T(b))\Big).$$

Since the support tripotent of $\varphi$, $s(\varphi)$, is a projection with $s(\varphi)\leq p\leq 1$, we have:
$$ \left\|(1-p) T(a)^* T(b)\right\|_{\varphi}^2 =\varphi \J {(1-p) T(a)^* T(b)}{(1-p) T(a)^* T(b)}{p} $$
$$=\varphi P_2 (p) \J {(1-p) T(a)^* T(b)}{(1-p) T(a)^* T(b)}{p} $$
$$=\varphi P^1 (p) \J {(1-p) T(a)^* T(b)}{(1-p) T(a)^* T(b)}{p} $$
$$= \frac12 \varphi \Big(((1-p) T(a)^* T(b)) ((1-p) T(a)^* T(b))^* p \Big)$$ $$+  \frac12 \varphi \Big(p ((1-p) T(a)^* T(b))^* ((1-p) T(a)^* T(b))\Big)$$
$$= \frac12 \varphi \Big( ((1-p) T(a)^* T(b))^* ((1-p) T(a)^* T(b))\Big) =0$$ for every $a,b\in A$ and $\varphi \in F$. Observing that $$P_2 (p) \J {(1-p) T(a)^* T(b)}{(1-p) T(a)^* T(b)}{p}$$ $$ = P^1 (p) \J {(1-p) T(a)^* T(b)}{(1-p) T(a)^* T(b)}{p}$$ is a positive element in $(B'')^1 (p)$ (see $(\ref{equation positive element})$) and $F= \{\varphi\in B' : \varphi (p) = \varphi (1) = \|\varphi\|\leq 1\}$ is the normal
quasi-state of the JBW-algebra $(B'')^1 (p)$, we have $P_2 (p) \J {(1-p) T(a)^* T(b)}{(1-p) T(a)^* T(b)}{p}= 0$ for every $a,b\in A$. This implies that
$$ p \Big((1-p) T(a)^* T(b) T(b)^* T(a) (1-p) + p T(b)^* T(a) (1-p) T(a)^* T(b)  \Big)^* p=0,$$ or equivalently, $$\Big((1-p) T(a)^* T(b) p\Big)^* \Big((1-p) T(a)^* T(b) p\Big)= p T(b)^* T(a) (1-p) T(a)^* T(b) p=0,$$ and hence \begin{equation}\label{eq 1 in prop linear contraction}
{(1-p) T(a)^* T(b)} p =  0
\end{equation} for every $a,b\in A$, which proves the second statement in $(a)$. By definition of $F$, for each $a,b\in A$ and $\varphi \in F$, we have $${\varphi} \Big((\Theta_{a,b}) (\Theta_{a,b})^* + (\Theta_{a,b})^* (\Theta_{a,b}) \Big)=0,$$ which gives $$0= {\varphi} P_2(p)\Big((\Theta_{a,b})^* (\Theta_{a,b}) \Big) = {\varphi} P^1(p)\Big((\Theta_{a,b})^* (\Theta_{a,b}) \Big).$$ Recalling that $F$ coincides with the normal
quasi-state of the JBW-algebra $(B'')^1 (p)$ and $P^1(p)\Big((\Theta_{a,b})^* (\Theta_{a,b}) \Big) = p (\Theta_{a,b})^* (\Theta_{a,b})p = (\Theta_{a,b} \ p)^* (\Theta_{a,b}\ p) $ is a positive element in  $(B'')^1 (p)$, we deduce that $\Theta_{a,b} \ p =0$, and accordingly $$T(\J aba) p = \J {T(a)}{T(b)}{T(a)} p.$$ On the other hand, $$\J {T(a)p}{T(b)p}{T(a)p} = T(a)pT(b)^*T(a)p$$ $$= (\hbox{by } (\ref{eq 1 in prop linear contraction})) = T(a)T(b)^*T(a)p= \J {T(a)}{T(b)}{T(a)} p,$$ and we establish $(b)$.

Finally, when $q$ is a another projection in $B''$ satisfying $(a)$ and $(b)$, it follows straightforwardly that $F(q) := \{\varphi \in \mathcal{Q} (B) : \varphi (q)  = \|\varphi\|\leq 1\}$ is a norm-closed face contained in $F_1$ and hence $F(q) \subseteq F$. It is known that, under these conditions, $q\leq p$ (see, for example \cite[Theorem 2.2]{Ne00}).
\end{proof}


The projection $p\in B''$ given by Proposition \ref{p linear contractions} is called the \emph{structure projection}
of $T$ and is denoted by $p_{T}$. Unfortunately, in some cases, the structure projection $p_T$ satisfies that $T(.) p_T$
reduces to zero, even under the hypothesis of $T$ being an isometry.
In the complex setting, Chu and Wong proved in \cite[Therorem 3.10]{chu2} that $T(.) p_T$ is an isometry
whenever $T: A \to B$ is an (complex) linear isometry between C$^*$-algebras with $A$ abelian. We shall see next
that the same statement is not valid in the setting of real C$^*$-algebras.

\begin{example}\label{examp real/complex type to ChuWong Thm 3.10}{\rm
Let $K$ be a compact Hausdorff space and let $C(K)_{\mathbb{R}}$
be the real abelian C$^*$-algebra of complex continuous functions on
$K$ regarded as a real Banach space. Let $C(K,M_2(\mathbb{R}))$ be the real
C$^*$-algebra of continuous $M_2(\mathbb{R})$-valued functions on $K$,
where $M_2(\mathbb{R})$ denotes the real C$^*$-algebra of $2\times2$
real matrices. \\ Define $T: C(K)_{\mathbb{R}}\rightarrow C(K,M_2(\mathbb{R}))$ by
\begin{equation*}
T(f) = \left(
\begin{array}{ccc}
\textrm{Re}(f) & \textrm{Im}(f) \\
0 & 0 \\
\end{array} \right)\;\;\;\;(f\in C(K)_{\mathbb{R}}).
\end{equation*}
Then $T$ is a real linear isometry. We have $T(f)^3 =T(f^3)$ for all $f\in C(K)_{\mathbb{R}}$,
but in general $T(\{f,g,h\})\neq\{T(f),T(g),T(h)\}$. We shall now show that
the structure projection of $T$ satisfies $p_{T} \leq \left(
              \begin{array}{cc}
                0 & 0 \\
                0 & 1 \\
              \end{array}
            \right)$, and in this case, $T(.) p_{T} =0$ (cf. \cite{apazoglou}).
Let $p$ be a projection satisfying the thesis of Proposition \ref{p
linear contractions}. Then
\begin{align*}
\{T(f),T(g),T(f)\}&=\left(\begin{array}{ccc}
\textrm{Re}(f) & \textrm{Im}(f)\\
0 & 0 \\
\end{array} \right)\left(\begin{array}{ccc}
\textrm{Re}(g) & 0 \\
\textrm{Im}(g) & 0 \\
\end{array} \right)\left(\begin{array}{ccc}
\textrm{Re}(f) & \textrm{Im}(f)\\
0 & 0 \\
\end{array} \right)\\
&=\left(\begin{array}{ccc}
\textrm{Re}^2(f)\textrm{Re}(g) & \textrm{Im}^2(f)\textrm{Im}(g)\\
+\textrm{Re}(f)\textrm{Im}(f)\textrm{Im}(g) & +\textrm{Re}(f)\textrm{Im}(f)\textrm{Re}(g)\\
0 & 0 \\
\end{array} \right)
\end{align*}
and
\begin{equation*}
T(\{f,g,f\})=\left(\begin{array}{ccc}
\textrm{Re}^2(f)\textrm{Re}(g) & \textrm{Im}^2(f)\textrm{Im}(g)\\
-\textrm{Im}^2(f)\textrm{Re}(g) & -\textrm{Re}^2(f)\textrm{Im}(g)\\
+2\textrm{Re}(f)\textrm{Im}(f)\textrm{Im}(g) & +2\textrm{Re}(f)\textrm{Im}(f)\textrm{Re}(g)\\
0 & 0 \\
\end{array} \right).
\end{equation*}
By assumptions:
\begin{align*}
&\left(\begin{array}{ccc}
\textrm{Re}^2(f)\textrm{Re}(g) & \textrm{Im}^2(f)\textrm{Im}(g)\\
+\textrm{Re}(f)\textrm{Im}(f)\textrm{Im}(g) & +\textrm{Re}(f)\textrm{Im}(f)\textrm{Re}(g)\\
0 & 0 \\
\end{array} \right)p\\
&=\left(\begin{array}{ccc}
\textrm{Re}^2(f)\textrm{Re}(g) & \textrm{Im}^2(f)\textrm{Im}(g)\\
-\textrm{Im}^2(f)\textrm{Re}(g) & -\textrm{Re}^2(f)\textrm{Im}(g)\\
+2\textrm{Re}(f)\textrm{Im}(f)\textrm{Im}(g) & +2\textrm{Re}(f)\textrm{Im}(f)\textrm{Re}(g)\\
0 & 0 \\
\end{array} \right)p
\end{align*}
which simplifies to
\begin{equation}\label{4.4.2b}
\left(\begin{array}{ccc}
\textrm{Re}(f)\textrm{Im}(f)\textrm{Im}(g) & \textrm{Re}(f)\textrm{Im}(f)\textrm{Re}(g)\\
-\textrm{Im}^2(f)\textrm{Re}(g) & -\textrm{Re}^2(f)\textrm{Im}(g)\\
0 & 0 \\
\end{array} \right)p=0.
\end{equation}
Let $f,g\in C(K)_{\mathbb{R}}$ be the constant functions $g(x)=1$ and $f(x)= \imath$.
Substituting $f$ and $g$ into (\ref{4.4.2b}) we
obtain
\begin{equation*}\label{4.4.3b}
\left(\begin{array}{ccc}
-1 & 0\\
0 & 0 \\
\end{array} \right) p(x) =0\qquad(x\in K),
\end{equation*}
which gives the desired statement.

According to the result established by Chu and Wong in \cite[Therorem 3.10]{chu2},
the mapping $T$ cannot be complexified to an
isometry on the complexification $C(K)_{\mathbb{R}}\oplus \imath C(K)_{\mathbb{R}}$.
Here is a direct argument: let $T_c$ denote the complexification of $T$.
\begin{align*}
T_c(g+if) &= T(g)+iT(f)\\
&=\left(
\begin{array}{ccc}
\textrm{Re}(g)+i\textrm{Re}(f) & \textrm{Im}(g)+i\textrm{Im}(f) \\
0 & 0 \\
\end{array} \right)\in C(K,M_2(\mathbb{C})).
\end{align*}
Let $g(x)=1$ and $f(x) = \imath$ be the constant functions considered before. Then
$\|g+if\|=2$  while $$ \|T_c(g+if)\| =\left\|\left(
\begin{array}{ccc}
1 & i  \\
0 & 0 \\
\end{array} \right)\right\|_{M_2(\mathbb{R})_c}=\sqrt{2}.$$
Therefore $T_c:C(K)_{\mathbb{R}} \bigoplus \imath C(K)_{\mathbb{R}} \rightarrow C(K,M_2(\mathbb{C}))$ is not an
isometry.
}
\end{example}

We shall see later, that the obstacle in the above example relies on the ``complex nature''
of $C(K)_{\mathbb{R}}$ and the real linearity of $T$.

Let $A$ be a real C$^*$-algebra, a \emph{``complex'' character} of
$A$ is a real linear homomorphism $\rho :A\to \mathbb{C}$. Since
every complex character $\rho: A\to \mathbb{C}$ admits a complex
linear extension to a character $\widetilde{\rho}: A_c \to
\mathbb{C}$, we may assume, via Gelfand theory, that $\rho$ is a
real linear $^*$-homomorphism. Every non-zero complex character of
$A$ is unital whenever $A$ has a unit element, or can be extended to
a unital complex character of the unitisation of $A$ otherwise.
Following standard notation, $\Omega(A)$ will denote the space of
all complex characters of $A$. We write $$\Omega (A)_{\mathbb{F}} :=
\{\rho \in \Omega(A) : \rho(A) = \mathbb{F}\},\hbox{ where }
\mathbb{F} = \mathbb{R} \hbox{ or }  \mathbb{C}.$$ The elements in
$\Omega (A)_{\mathbb{R}}$ (resp., in $\Omega (A)_{\mathbb{C}}$) are
called \emph{complex characters of real type} (resp., \emph{complex
characters of complex type}). When $A$ is abelian, we shall say that
$A$ is of \emph{real type} if $\Omega (A)_{\mathbb{C}} = \emptyset$,
and of \emph{complex type} if $\Omega (A)_{\mathbb{C}} = \emptyset$.
An abelian real C$^*$-algebra $A$ is of real type (resp., of complex
type) if and only if all non-zero {complex characters} of $A$ are
onto $\mathbb{R}$ (resp., onto $\mathbb{C}$) (compare \cite[Theorem
2.7.7]{li1}). In the more general setting of real commutative
J$^*$-algebras, real and complex types where studied by L.J. Bunce
and C.-H. Chu in \cite[\S 3]{BuChu}. When $A$ is an abelian complex
C$^*$-algebra, then the real C$^*$-algebra underlying $A$,
$A_{\mathbb{R}}$, is a real C$^*$-algebra of complex type. Given a
locally compact Hausdorff space $X$, the real C$^*$-algebra
$C_0(X,\mathbb{R})$ is of real type. Furthermore, it follows from
\cite[Lemma 3.4]{BuChu} that every abelian real C$^*$-algebra of
real type if of the form $C_0(X,\mathbb{R})$, for some locally
compact Hausdorff space $X$.

Our next result shows that the main result in \cite{chu2} can be also
proved when the domain is an abelian real C$^*$-algebra of real type.

Henceforth, the closed unit ball of a Banach space $E$
will be denoted by $E_1$ and the set of extreme points of a convex set $S$ will be denoted by
$\partial S$.


\begin{proposition}\label{p extreme points} Let $T: X \to Y$ be a (not necessarily surjective) linear isometry
between Banach spaces and let $E=T(X)$. For each $\psi \in \partial E^{\prime}_1,$ the set
$$\mathcal{Q}_{\psi} = \left\{ \phi \in \partial Y^{\prime}_1 : \phi|_{E} =\psi \right\}$$ is
non-empty, that is, there exists $\phi \in \partial Y^{\prime}_1$ satisfying $\psi=\phi|_E$ and
$T'(\phi)\in \partial X^{\prime}_1$.
\end{proposition}

\begin{proof} Let $T':E'\rightarrow X'$ be
the dual map of the surjective isometry $T:X\rightarrow E$. Then
$T'$ is a linear surjective isometry. Let $\psi$ be an extreme
point of the unit ball, $E_1'$, of $E'$. By the Hahn-Banach theorem,
the set $$\mathcal{F}_{\psi}=\left\{ \phi\in Y^{\prime}_{1} : \phi|_{E} =
 \psi\right\}$$ is a non-empty weak$^*$ closed face of
$Y^{\prime}_1$, thus, by the Krein-Milman theorem, there exists
$\phi\in \mathcal{F}_{\psi} \bigcap \partial Y^{\prime}_1 = \mathcal{Q}_{\psi}$. By an abuse of notation, we also denote
by $T^{\prime} : E^{\prime} \to X^{\prime}$ the dual map of $T : X
\to Y$ since confusion is unlikely. One sees readily that
$T'(\phi) = T'(\psi)$. Since $\psi=\phi|_E$ is an extreme point in $E_1'$, $T'(\psi)$ is an extreme point in $X'_{1}$ because $T':E'\rightarrow X'$ is a surjective linear isometry.
\end{proof}

We show below that if a real C$^*$-algebra $A$ admits a complex
character of real type (i.e. $\Omega(A)_{\mathbb{R}}\neq \emptyset$)
then the mapping $T(\cdot) p_T$ is non-zero for every linear
isometry from $A$ into another C$^*$-algebra. Our next result is in
fact an appropriate real version of \cite[Proposition 2]{chu9} and
\cite[Proposition 4.3]{chu2}.

\begin{theorem}\label{t Chu Wong Prop 4.3} Let $T: A \to B$ be a (not necessarily surjective) linear isometry between real C$^*$-algebras. Suppose $\rho \in \Omega_{\mathbb{R}} (A)$, then there exists
$\phi\in \partial B'_1$ and a minimal partial isometry $u_\phi\in B''$ such that $T'(\phi) = \rho$,
$u_\phi$ is the support partial isometry of $\phi$, ${T''(x)} = \rho (x) {u_{\phi}} + P_0 (u_{\phi}) (T''(x)),$
$$\{u_{\phi},u_{\phi},\{T''(x),T''(y),T''(z)\}\} =
\{u_{\phi},u_{\phi},T''(\{x,y,z\})\},$$
and
$$ \{u_{\phi},\{T''(x),T''(y),T''(z)\},u_{\phi}\} =
\{u_{\phi},T''(\{x,y,z\}),u_{\phi}\},$$ for every $x,y,z\in A''.$ Moreover, the minimal projection $p_\phi=u_\phi^*u_\phi$ satisfies statements $(a)$ and $(b)$ in Proposition \ref{p linear contractions}
and the mapping $T(\cdot) p_{\phi}: A \to B''$ is a non-zero triple homomorphism. In particular,
$T(\cdot) p_{T}$ is non-zero.
\end{theorem}

\begin{proof}
Let $\rho\in \Omega(A)_{\mathbb{R}}$, where $A$ is a (not necessarily abelian) real C$^*$-algebra. The map $\rho'': A'' \to \mathbb{R}$ is a surjective norm-one weak$^*$ continuous *-homomorphism. The set $K= \ker (\rho'')$ is a weak$^*$ closed ideal of $A''$, thus there exists a minimal and reduced central projection $p\in A''$ such that $A'' = K\oplus^{\infty} \mathbb{R} p$ and $K= A'' (1-p).$ In particular, $\rho\in \partial A'_1.$

If $T:A\to B$ is a linear isometry, taking $E=T(A),$ $T': E' \to A'$ is a surjective linear isometry, then there exists $\psi\in \partial E'_1$ such that $T'(\psi) = \rho$. Proposition \ref{p extreme points} assures the existence of $\phi\in \partial B'_1$ such that $T' (\phi) = T'(\psi) = \rho$. We observe that $(A'')^{-1} (p) = \{0\},$ $P^{1} (p) (x) = x p$ and  $T'' : A'' \to B''$ is a weak* continuous linear isometry.

By the remarks following Lemma \ref{lemma3js} (cf. \cite[Corollary 2.1]{peralta_stacho}), the minimal
tripotents of $B''$ are the support tripotents of the extreme points
of $B'_1$. Let $u_{\phi}\in B''$ be the support
tripotent of $\phi$. Then $\phi (.) = \phi Q(u_{\phi}) (.)$
and $P^{1} (u_{\phi}) (.) = \phi (.) u_{\phi}$ (cf. \cite[Lemma 2.7]{peralta_stacho}). Thus, $\phi T(x)
=T^{\prime} (\phi) (x) = \rho (x)$ implies
\begin{equation}\label{correct eq 3.1} P^{1} (u_{\phi}) T(x) = \phi(T(x)) u_{\phi} = \rho (x) u_{\phi},
\end{equation} for every $x\in A$. Since $T'': A''\to B''$ and $\rho\in A'$ are weak* continuous and $A_1$ is weak* dense in $A''_1$, we have \begin{equation}\label{correct eq 3.1b} P^{1} (u_{\phi}) T''(x) = \phi(T''(x)) u_{\phi} = \rho (x) u_{\phi},
\end{equation} for every $x\in A''$. In particular, $$P^{1} (u_{\phi}) T''(\{ x,y,z\} ) = \rho \J xyz u_{\phi} =
\rho(x) \rho(y) \rho(z) \{u_{\phi},u_{\phi},u_{\phi}\} $$ $$= \{P^{1}
(u_{\phi}) T''(x),P^{1} (u_{\phi}) T''(y),P^{1} (u_{\phi}) T''(z)\},$$ for
every $x,y,z\in A''$, that is, $P^{1} (u_{\phi}) T'': A'' \to
B''$ is a non-zero triple homomorphism.

We shall show that $\{u_\phi,u_\phi,T(a)\}= u_\phi$ for every $a\in A''$
satisfying $\|a\|=1$ and $P^1(p) (a)= p$. By $(\ref{correct eq 3.1b})$ and the Cauchy-Schwarz
inequality (cf. Lemma \ref{lemma1js}), we have
\begin{align*}
1&=\|P^1 (p) (a) \|= |\phi\circ T''(a)|^2=|\phi(\{u_\phi,T''(a),u_\phi\})|^2\\
&\leq \phi(\{u_\phi,u_\phi,u_\phi\})\phi(\{T''(a),T''(a),u_\phi\})\\
&\leq\|T''(a)\|^2=\|a\|^2=1
\end{align*}
which implies \begin{equation*}\phi(\{T''(a),T''(a),u_\phi\})=1.\end{equation*}
Let $N_\phi=\{b\in B''\;:\;\phi(\{b,b,u_\phi\})=0\}$. By Lemma \ref{p0space},
\begin{equation}\label{nphi}
N_\phi=P_0(u_\phi)(B'').
\end{equation}

Now we claim that $T''(a)-u_\phi\in N_\phi$, for any $a$ as above.
Indeed, taking into account that, by Lemma
\ref{lemma1js},  $\phi(\{T''(a),u_\phi,u_\phi\})=\phi(\{u_\phi,T''(a),u_\phi\})$, we have
\begin{align*}
&\phi(\{T''(a)-u_\phi,T''(a)-u_\phi,u_\phi\})\\
&=\phi(\{T''(a),T''(a),u_\phi\})-2 \phi(\{ u_\phi,T''(a),u_\phi\})+\phi(u_\phi)\\
&=\phi(\{T''(a),T''(a),u_\phi\})-2 \phi(T''(a))+\phi(u_\phi)\\
&=1-2\rho (a)+1=0.
\end{align*}
Therefore, by (\ref{nphi}), $T''(a)- u_\phi\in
P_0(u_\phi)(B'')$ and then \begin{equation}
\label{eq new theorem 1} \{u_\phi,u_\phi,T''(a)\}=u_\phi,
\end{equation} for every $a$ as above.

In the next step we prove that $\phi(\{T''(b),T''(b),u_\phi\})=0$ whenever $b\in A''$
satisfies $P^1 (p) (b)=0$. Without loss of generality we can take
$\|b\|=1$. Since $A'' = A''(1-p) \oplus^\infty \mathbb{R} p$, $p$ and $b$ are orthogonal in $A''$,
$\|p+b\| = 1$ and $P^1 (p) (b+p)= p$. We deduce, from the above arguments, that $T''(b+p)+
N_\phi=u_\phi+N_\phi=T''(p)+N_\phi,$ which implies $T''(b)\in N_\phi$ and
$\phi(\{T''(b),T''(b),u_\phi\})=0$.

Now, let $c\in A''$ with $\|c\|=1$. Then $P^1 (p) (c-\rho (c) p)=0$
and, by the arguments of the previous paragraph, we have
$T''(c-\rho (c) p)\in N_\phi$. That is, $\{u_\phi,u_\phi,T''(c-\rho (c) p)\}=0$ and hence, by $(\ref{eq new theorem 1})$,
\begin{equation}\label{4.3.4}
\{u_\phi,u_\phi,T''(c)\}=\rho(c) \{u_\phi,u_\phi,T''(p)\}=  \rho(c) u_\phi.
\end{equation} Left multiply by $u_\phi u_\phi^*$, right multiply by $u_\phi^* u_\phi$ and subtracting, we have $u_\phi u_\phi^* T''(c) = T''(c) u_\phi^* u_\phi = u_\phi u_\phi^* T''(c) u_\phi^* u_\phi$, for every $c\in A''$.

Therefore \begin{equation} \label{claim 4} {T''(c)} = \rho (c)
{u_{\phi}} + P_0 (u_{\phi}) (T''(c)),
\end{equation} for every $c\in A''$. It follows that $$T''(\{x,y,z\}) =  \rho \{x,y,z\} \ u_{\phi} + P_0 (u_{\phi}) (T''(\{x,y,z\} )),$$ and by Peirce rules $$\{T''(x),T''(y),T''(z)\} =$$ $$= \rho (x) \rho (y) \rho (z) u_{\phi} + \{ {P_0 (u_{\phi}) (T''(x))},{P_0 (u_{\phi}) (T''(y))},{P_0 (u_{\phi})
(T''(z))}\}$$
$$= \rho \{ x,y,z \} u_{\phi} + \{ {P_0 (u_{\phi}) (T''(x))},{P_0 (u_{\phi}) (T''(y))},{P_0 (u_{\phi})
(T''(z))}\} ,$$
 which assures that
\begin{equation}
\label{eq old 3.13} \{u_{\phi},u_{\phi},\{T''(x),T''(y),T''(z)\}\} =
\{u_{\phi},u_{\phi},T''(\{x,y,z\})\},
\end{equation}
and
\begin{equation}
\label{eq old 3.13b} \{u_{\phi},\{T''(x),T''(y),T''(z)\},u_{\phi}\} =
\{u_{\phi},T''(\{x,y,z\}),u_{\phi}\},
\end{equation} for every $x,y,z\in A''.$

Let $p_\phi=u_\phi^*u_\phi$ and $q_\phi=u_\phi u_\phi^*$ be the
initial and final minimal projections of the minimal tripotent
$u_\phi$ (see Lemma \ref{lemma2js}). We then have
$$(\{T''(x),T''(y),T''(z)\} - T''(\{x,y,z\}))p_\phi+$$
$$q_\phi(\{T''(x),T''(y),T''(z)\}-T''(\{x,y,z\})) = 0,$$ and
$$u_{\phi} (\{T''(x),T''(y),T''(z)\}-T''(\{x,y,z\}))^* u_\phi=0,$$
and consequently:
\begin{equation*}
q_{\phi} (\{T''(x),T''(y),T''(z)\}-T''(\{x,y,z\})) p_\phi=0,
\end{equation*}
and
\begin{equation}\label{4.3.13}
(\{T''(x),T''(y),T''(z)\}-T''(\{x,y,z\})) p_\phi=0,
\end{equation} witnessing that $T''(\cdot) p_{\phi} : A \to B''$ is a non-zero
triple homomorphism (observe that $T''(\cdot) p_{\phi} =\rho (\cdot) u_{\phi} \neq 0$).
Proposition \ref{p linear contractions} implies that $p\leq p_T$ and $T(\cdot) p_{T} \neq 0$.
\end{proof}

Suppose $A$ is an abelian real C$^*$-algebra of complex type, that is, $\Omega(A)_{\mathbb{R}}= \emptyset$. In general, $A$ need not be C$^*$-isomorphic to $C_0(X)_{\mathbb{R}}$ (compare, for example, \cite[Remark 3.8]{BuChu}). However, by \cite[Theorem 3.7]{BuChu}, $A''$ is isometric and C$^*$-isomorphic to $C(\Omega)_{\mathbb{R}}$, for some compact hyperstonean space $\Omega$. Let $J_A : A \hookrightarrow A''$ denote the canonical inclusion of $A$ into $A''$, then $J_A$ is an isometric C$^*$-embedding. Let $T: A''=C(\Omega)_{\mathbb{R}}\rightarrow C(\Omega,M_2(\mathbb{R}))$
be the isometry given in Example \ref{examp real/complex type to ChuWong Thm 3.10}. That is
\begin{equation*}
T(f) = \left(
\begin{array}{ccc}
\textrm{Re}(f) & \textrm{Im}(f) \\
0 & 0 \\
\end{array} \right)\;\;\;\;(f\in C(\Omega)_{\mathbb{R}}).
\end{equation*} The mapping $S= T J_A : A \to C(\Omega,M_2(\mathbb{R}))$ is an isometry satisfying $S(.) p_S =0$.

When $T$ is an isometry from an abelian real C$^*$-algebra of
real type to another real C$^*$-algebra, we can prove that $T(.) p_{T}$ is an isometry.

\begin{theorem}\label{theorem1iso}
Let $B$ be a real C$^*$-algebra and $T:C_0(X,\mathbb{R})\rightarrow B$
be a (not necessarily surjective) linear isometry. Let $p_T \in  B''$
be the structure projection of $T$. Then 
$T(.)p_T : C_0(X,\mathbb{R})\rightarrow B''$
is an isometric triple homomorphism.
\end{theorem}

\begin{proof}
Let us denote $A=C_0(X,\mathbb{R})$ and $E=T(A)$ and let $T':E'\rightarrow A'$ be
the dual map of the surjective isometry $T:A\rightarrow E$. The real C$^*$-algebra
$A$ is of real type and $\Omega(A)_{\mathbb{R}} = \{ \delta_x : x\in X\}$.
By Proposition \ref{p extreme points}, the set
$$\mathcal{Q} = \left\{ \phi \in \partial B^{\prime}_1 : \phi|_{E} \in \Omega(A)_{\mathbb{R}}\right\}$$ is
non-empty. By Theorem \ref{t Chu Wong Prop 4.3}, for each $x\in X$, there exists $\phi_x\in \mathcal{Q}$ and a minimal partial isometry $u_x\in B''$ such that $T'(\phi_x) = \delta_x$,
$u_x$ is the support partial isometry of $\phi_x$, \begin{equation}\label{4.3.13b} {T(f)} = \delta_x (f) {u_{x}} + P_0 (u_{x}) (T(f)),
\end{equation} $$
(\{T(f),T(g),T(h)\}-T(\{f,g,h\})) p_x=0,$$ and $$p_x T(f)^* T(g) = T(f)^* T(g) p_x,$$
for every $f,g,h\in A$, where $p_x=u_x^* u_x$ is a minimal projection in $B''$.
Let $p=\bigvee_{x\in X} p_x$ be the lattice supremum in $B''$.
By Lemma \ref{lemma3js} and (\ref{4.3.13b}), we obtain
\begin{equation}\label{secondresult}
\{T(f),T(g),T(h)\}p=T(\{f,g,h\})p \hbox{ and } p T(f)^* T(g) = T(f)^* T(g) p,
\end{equation} for every $f,g,h\in A$. Thus,
$T(\cdot) p : A \to B''$ is a triple homomorphism. Proposition \ref{p linear contractions}
implies $p\leq p_T.$

We claim that $T(\cdot) p$ is an isometry. Indeed, for each $f\in A$,
we can pick $x_0\in X$ with $\|f\|=|f(x_0)|$. Let $\phi_{x_0}\in \mathcal{Q}$
be such that $T'(\phi_{x_0})=\delta_{x_0}$. In this case,
$$\|f\|=\|T(f)\| \geq\|T(f)p\|\geq\|T(f)pp_{x_0}\|=\|T(f)p_{x_0}\|$$
$$=\|T(f)u_{x_0}^*u_{x_0}\| \geq\|u_{x_0} u_{x_0}^*T(f)u_{x_0}^*u_{x_0}\|$$ $$=\|\{u_{x_0},\{u_{x_0},T(f),u_{x_0}\},u_{x_0}\}\|
=\|f(x_{x_0})u_{x_0}\|=\|f\|,$$ which implies $\|T(f)p\|=\|f\|$ for all $f\in A$.

Finally, since $p\leq p_T$ we have $$\|f\|= \|T(f) p \| = \|T(f) p_T p \|\leq \|T(f) p_T \|\leq \|T(f)\|= \|f\|,$$ which completes the proof.
\end{proof}

Let $E$ be a real or complex JBW$^*$-triple, two elements $\phi,\psi\in E_{_{'}}$ are said to be
\emph{orthogonal} (written $\phi\perp \psi$) if their support tripotents are orthogonal, that is, $u_\phi \perp u_{\psi}$. C.M. Edwards and G. Rüttimann proved in \cite[Theorem 5.4]{EdRu01} that in the setting of complex JBW$^*$-triples, $\phi$ and $\psi$ are orthogonal if, and only if, $\|\phi \pm \psi\|  = \|\phi\| + \|\psi\|$.

Suppose $E$ is a real JBW$^*$-triple whose complexification, $E_c$, is a (complex) JBW$^*$-triple.
We have already commented that there is a conjugate linear isometry
$\tau:E_c\rightarrow E_c$ of period 2 such that $E = E_c^{\tau} =\{b\in
E_c\::\:\tau(b)=b\}$. The ``dual'' map
$\widetilde{\tau}:E_c{^*}\rightarrow E_c^{*}$ defined by
\begin{equation*}
\widetilde{\tau}(\phi)(b)=\overline{\phi(\tau(b))}\qquad(\phi \in E_c^*,\,b\in E_c)
\end{equation*} is a conjugate linear isometry of period 2 and the mapping $$(E_c^*)^{\widetilde{\tau}} \to E'$$ $$\varphi \mapsto \varphi|_{E}$$ is a surjective linear isometry. Moreover, by \cite[Proposition 2.3]{MarPe}, $\tau$ is weak* continuous and hence $\widetilde{\tau}((E_c)_{*}) = (E_c)_{*}$ and the restricted mapping $$((E_c)_*)^{\widetilde{\tau}} \to E_{_{'}}$$ $$\varphi \mapsto \varphi|_{E}$$ is a surjective linear isometry. Take two elements $\phi,\psi\in E_{_{'}}$, then $\phi\perp \psi$ (in $E_{_{'}}$) if, and only if, $\phi\perp \psi$ as elements in $E_*$, therefore $\|\phi\pm \psi\| = \|\phi\| +\|\psi\|$ as elements in $E_*$, which is equivalent to $\|\phi\pm \psi\| = \|\phi\| +\|\psi\|$ as elements in $E_{_{'}}$.

Let $\mathcal{U} (E)$ and $\mathcal{U} (E_c)$ denote the sets of all
tripotents in $E$ and $E_c$, respectively. Let $\mathcal{U}
(E_c)^{\sim}$ denote the union of the set $\mathcal{U} (E_c)$ and a
one point set $\{\omega\}$ such that $e\leq \omega$ for every
$e\in\mathcal{U} (E_c)$. It is known that $\mathcal{U} (E_c)^{\sim}$
is a complete lattice and $\mathcal{U} (E)\cup \{\omega\}= \{ e\in
\mathcal{U} (E_c) : \tau (e) = e\}\cup \{\omega\}$ is a sub-complete
lattice of $\mathcal{U} (E_c)^{\sim}$. The supremum of a family
$(e_i) \subset \mathcal{U} (E_c)$ (resp., in $\mathcal{U} (E)$) need
not exist, in general, in $\mathcal{U} (E_c)$ (resp., in
$\mathcal{U} (E)$). However, for every family $(e_i)$ of mutually
orthogonal elements in $\mathcal{U} (E_c)$ (resp., in $\mathcal{U}
(E)$), the supremum $\bigvee_{i} e_i$ exists in $\mathcal{U} (E_c)$
or in (resp., in $\mathcal{U} (E)$) (cf. \cite[Theorem 5.1]{EdRu01}
and \cite{edwards_ruttimann}). Moreover, suppose that an element
$a\in E$ satisfies $a \perp e_i$ for every $i$, then $a\perp
\bigvee_{i} e_i$.

The next lemma subsumes some of the above results.

\begin{lemma}\label{l supremum} Let $E$ be a real JBW$^*$-triple. The following statements hold:
\begin{enumerate}[$(a)$] \item Let $\phi$ and $\psi$ be two elements in $E_{_{'}}$, then $\phi\perp \psi$ if, and only if, $\|\phi\pm \psi\| = \|\phi\| +\|\psi\|$.
\item Let $(e_i)$ be a family of mutually orthogonal elements in $\mathcal{U} (E)$. Then the supremum $\bigvee_{i} e_i$ exists in $\mathcal{U} (E)$. Moreover, suppose $a\in E$ satisfies $a\perp e_i$, for every $i$, then $\bigvee_{i} e_i\perp a$. $\hfill\Box$
\end{enumerate}
\end{lemma}

The next result reveals the connection between the real versions of Theorem 1 in
\cite{chu9} and Theorem 3.10 in \cite{chu2}. 

\begin{theorem}\label{theorem2iso}
Let $B$ be a real JB$^*$-triple and $T:C_0(X,\mathbb{R})\rightarrow B$
be a not necessarily surjective linear isometry. Then there exists
a partial isometry $u\in B''$ such that $${T(f)} = \delta_x (f) {u_{x}} + P_0 (u_{x}) (T(f)),$$
$$\{u,T(\{f,g,h\}),u\}=\{u,\{T(f),T(g),T(h)\},u\},$$
for every
$f,g,h\in C_0(X,\mathbb{R})$, and $\{u,T(\cdot),u\} :C(X,\mathbb{R})\rightarrow B''$ is an isometry.
\end{theorem}

\begin{proof} Keeping in mind the notation in the proof of Theorem \ref{theorem1iso}, we write $A=C_0(X,\mathbb{R})$, $\Omega(A)_{\mathbb{R}}=\{\delta_x : x\in X\}$ and $$\mathcal{Q} = \left\{ \phi \in \partial B^{\prime}_1 : \phi|_{E} \in \Omega(A)_{\mathbb{R}}\right\}.$$
By Theorem \ref{t Chu Wong Prop 4.3}, whose proof is valid when $B$ is a real JB$^*$-triple, for each $x\in X$, there exists $\phi_x\in \mathcal{Q}$ and a minimal partial isometry $u_x\in B''$ such that $T'(\phi_x) = \delta_x$,
$u_x$ is the support partial isometry of $\phi_x$, \begin{equation}\label{4.3.13c} {T(f)} = \delta_x (f) {u_{x}} + P_0 (u_{x}) (T(f)),
\end{equation} and hence \begin{equation*} T(\{f,g,h\})- \{T(f),T(g),T(h)\}\in B''_0 (u_{x})
\end{equation*} for every $f,g,h\in A$ and $x\in X$.

We shall prove now that $u_x\perp u_y$ whenever $x\neq y$ in $X$. Indeed, for $x\neq y$ we have $\delta_x \perp \delta_y$, and thus
$$2= \|\delta_x \pm \delta_y\| = \|T'(\phi_x\pm \phi_y)\|\leq \|\phi_x\pm \phi_y\|\leq \|\phi_x\| +\| \phi_y\| =2. $$ It follows from Lemma \ref{l supremum}$(a)$ and the comments preceding it, that $\phi_x\perp \phi_y$, or equivalently, $u_x\perp u_y$. Therefore $(u_x)_{x\in X}$ is a family of mutually orthogonal minimal tripotents in $B''$. The supremum $u =  \bigvee_{x} u_x$ exists and defines a tripotent in $B''$ (cf. Lemma \ref{l supremum}$(b)$). Since, by $(\ref{4.3.13c})$, for each $f,g,$ and $h$ in $A$ and $x\in X$, we have $T(\{f,g,h\})- \{T(f),T(g),T(h)\}\perp u_{x}$, we deduce from Lemma \ref{l supremum}$(b)$, that $T(\{f,g,h\})- \{T(f),T(g),T(h)\}\perp u,$ for every $f,g,h\in A$.

In order to prove the last statement, observe that for each
$f\in A$ we can take $x_0\in X$ with $\|f\|=|f(x_0)|$. Let $\phi_{x_0}\in \partial
B'_1$ such that $T'(\phi_{x_0})=\delta_{x_0}$. Then,
\begin{align*}
\|f\|=\|T(f)\|&\geq\|\{u,T(f),u\}\|\geq\|\{u_0,\{u,T(f),u\},u_0\}\|\\
&\geq\|Q(u_{x_0}) Q(u)(T(f)) \|=\|Q(u_{x_0})^2 (T(f))\|\\
&=\|\{u_0,\{u_0,T(f),u_0\},u_0\}\|=\|f(x_{x_0})u_{x_0}\|=\|f\|,
\end{align*} which proves
$$\|\{u,T(f),u\}\|=\|f\|\qquad (f\in A).$$
\end{proof}

\begin{remark}\label{r last} When in the above Theorem \ref{theorem2iso}, $B$ is a real C$^*$-algebra,
the projection $p= u^* u$ coincides with the one considered in the proof of Theorem \ref{theorem1iso}.
\end{remark}

Using the previous results, we can now show that if $T$ is an
isometry between two real C$^*$-algebras $A$ and $B$, then $T$ is a
local triple homomorphism via a tripotent in $B''$. That does not
however imply that $T$ is a triple homomorphism on the whole algebra
$A$ (compare Example \ref{examp real/complex type to ChuWong Thm 3.10}). The following theorem is an
extension of the results in \cite{chu9,chu2}.

We recall that given an element $a$ in a real JB$^*$-triple $A$, the real
JB$^*$-subtriple, $A_a,$ generated by $a$ is linearly isometric to the
real JB$^*$-triple $C_0(X,\mathbb{R})$ of real-valued continuous functions on
$X$ vanishing at infinity (see \cite {chu0} and \cite{kaup1}).

\begin{theorem}\label{t 2iso}
Let $T:A\rightarrow B$ be a not necessarily surjective
linear isometry between two real JB$^*$-triples.
Then, for each $a\in A$, there exists a tripotent $u\in B''$ such that
\begin{enumerate}[$(a)$]
\item $\{u,T(\{f,g,h\}),u\}=\{u,\{T(f),T(g),T(h)\},u\}$, for all $f,g,h$ in the real JB$^*$-subtriple generated by $a$;
\item The mapping $\{u,T(\cdot),u\} :A_a\rightarrow B''$ is a linear isometry.
\end{enumerate}
Furthermore, when $B$ is a real C$^*$-algebra, the projection $p= u^* u$ satisfies that
$T(\cdot)p :A_a\rightarrow B''$ is an isometric triple homomorphism.
\end{theorem}

\begin{proof}
Since $A_a$ is linearly isometric to the JB$^*$-triple
$C_0(X,\mathbb{R})$ of all real-valued continuous functions on $X$
vanishing at infinity (cf. \cite{kaup1}), the result follows
immediately from Theorem \ref{theorem1iso}, Theorem
\ref{theorem2iso} and Remark \ref{r last}.
\end{proof}

\textbf{Acknowledgements:} The useful comments and suggestions made by the Referee in his/her report were fundamental to improve the presentation and quality of the paper. We are grateful to the referee for his/her thorough suggestions.

\bibliographystyle{amsplain}

\end{document}